\documentclass[11pt]{amsart}
\usepackage[margin=35mm]{geometry}
\usepackage{amsmath,amssymb}
\usepackage{amsthm}

\newtheorem{thm}{Theorem}[section]
\newtheorem{prop}{Proposition}[section]
\newtheorem{lem}{Lemma}[section]
\newtheorem{cor}{Corollary}[section]
\newtheorem{defi}{Definition}[section]

\begin{document}

\title{On the topological stability and shadowing in zero-dimensional spaces}
\author{Noriaki Kawaguchi}
\address{Graduate School of Mathematical Sciences, The University of Tokyo, 3-8-1 Komaba Meguro, Tokyo 153-8914, Japan}
\email{knoriaki@ms.u-tokyo.ac.jp}

\maketitle

\markboth{NORIAKI KAWAGUCHI}{On the topological stability and shadowing in zero-dimensional spaces}

\begin{abstract}
In this paper, we examine the notion of topological stability and its relation to the shadowing properties in zero-dimensional spaces. Several counter-examples on the topological stability and the shadowing properties are given. Also, we prove that any topologically stable (in a modified sense) homeomorphism of a Cantor space exhibits only simple typical dynamics.
\end{abstract}

\section{Introduction}

Stability is one of the most important notions in the qualitative study of dynamical systems, and the relationship with it is a basic subject of the theory of shadowing properties. The topological stability introduced by Walters in \cite{W} is a kind of structural stability defined for all homeomorphisms of compact metric spaces. Among the papers dealing with the notion, there are \cite{H, N, Y} (see also \cite{P}), and some recent attempts were made to explore the variants of topological stability \cite{AM, LM, MMT}. Although there are a number of studies on the topological stability, the main attention seems to have been focused on the homeomorphisms (or diffeomorphisms) of topological (or differentiable) manifolds. In this paper, we examine the notion in zero-dimensional spaces, mainly in relation with the shadowing properties, and observe some singular behaviors possibly different from when the spaces are manifolds. Especially, we give several counter-examples on the topological stability and the shadowing properties. Moreover, we prove that all topologically stable (in a modified sense) homeomorphisms of Cantor spaces exhibit only simple typical dynamics.

Let us begin with the definition of the topological stability. Throughout this paper, $(X,d)$ denotes a compact metric space $X$ endowed with a metric $d$. We denote by $\mathcal{C}(X)$ the set of continuous self-maps of $X$ and by $\mathcal{H}(X)$ the set of homeomorphisms of $X$. We define metrics $d_{C^0}$ on  $\mathcal{C}(X)$ and $D$ on $\mathcal{H}(X)$ by \[d_{C^0}(f,g)=\sup_{x\in X}d(f(x),g(x))\] for $f,g\in\mathcal{C}(X)$, and \[D(f,g)=\max\{d_{C^0}(f,g),d_{C^0}(f^{-1},g^{-1})\}\] for $f,g\in\mathcal{H}(X)$. As $\mathcal{H}(X)\subset\mathcal{C}(X)$, $d_{C^0}$ gives another metric on $\mathcal{H}(X)$, but we know that $d_{C^0}$ and $D$ are equivalent metrics on $\mathcal{H}(X)$. We give a proof of this fact in Section 2 for the sake of completeness. Then, for $f\in\mathcal{H}(X)$, we say that $f$ is {\em topologically stable} if for any $\epsilon>0$, there is $\delta>0$ such that for every $g\in\mathcal{H}(X)$ with $D(f,g)<\delta$, there is $h\in\mathcal{C}(X)$ with $d_{C^0}(h,id_X)<\epsilon$ and $h\circ g=f\circ h$.

The most part of this paper is concerned with the relation between the topological stability and the shadowing properties, so we shall recall the definitions of shadowing properties dealt with in this paper. Given a map $f:X\to X$, a finite sequence $(x_i)_{i=0}^{k}$ of points in $X$, where $k$ is a positive integer, is called a {\em $\delta$-chain} of $f$ if $d(f(x_i),x_{i+1})\le\delta$ for every $0\le i\le k-1$. A $\delta$-chain $(x_i)_{i=0}^{k}$ of $f$ is said to be a {\em $\delta$-cycle} of $f$ if $x_0=x_k$. For $\delta>0$, a sequence $(x_i)_{i\ge0}$ of points in $X$ is called a {\em $\delta$-pseudo orbit} of $f$ if $d(f(x_i),x_{i+1})\le\delta$ for all $i\ge0$. Then, for $\epsilon>0$, a $\delta$-pseudo orbit  $(x_i)_{i\ge0}$ of $f$ is said to be {\em $\epsilon$-shadowed} by $x\in X$ if $d(x_i,f^i(x))\leq \epsilon$ for all $i\ge 0$. We say that $f$ has the {\em shadowing property} if for any $\epsilon>0$, there is $\delta>0$ such that every $\delta$-pseudo orbit of $f$ is $\epsilon$-shadowed by some point of $X$. The following definition is not so standard as the shadowing property. We say that $f$ has the {\em strict periodic shadowing property} if for any $\epsilon>0$, there is $\delta>0$ such that for any $\delta$-cycle $(x_i)_{i=0}^m$ of $f$, where $m$ is a positive integer, there is $p\in X$ such that $f^m(p)=p$ and $d(x_i,f^i(p))\le\epsilon$ for all $0\le i\le m$. Let $Per(f)$ denote the set of periodic points for $f$. If we weaken the condition `$f^m(p)=p$' to `$p\in Per(f)$', it become the definition of the {\em periodic shadowing property} (see, for example, \cite{PS}). A point $x\in X$ is said to be a {\em chain recurrent point} for $f$ if for any $\delta>0$, there is a $\delta$-cycle $(x_i)_{i=0}^{k}$ of $f$ with $x_0=x_k=x$. We denote by $CR(f)$ the set of chain recurrent points for $f$. Note that  we have $CR(f)=\overline{Per(f)}$ when $f$ has the (strict) periodic shadowing property.

There is a property of compact metric spaces derived from \cite[Lemma 4]{SS}, which relates the topological stability to the shadowing properties. In \cite{W}, by using the fact that the closed differentiable manifolds of dimension grater than $1$ have such a property \cite[Lemma 10]{W} (see also \cite[Lemma 13]{NS}), Walters proved that all topologically stable homeomorphisms of those spaces satisfy the shadowing property. Its definition is given as follows. For $n\ge1$, an $n$-tuple $(x_1,x_2,\dots,x_n)\in X^n$ is said to be proper if $x_1,x_2,\dots,x_n$ are pairwise distinct, that is, $x_i\neq x_j$ for all $1\le i<j\le n$. For $n\ge 1$, we define a metric $d_n$ on $X^n$ by \[d_n(\zeta,\eta)=\max_{1\le i\le n}d(x_i,y_i)\] for $\zeta=(x_1,x_2,\dots,x_n)$, $\eta=(y_1,y_2,\dots,y_n)\in X^n$. For any map $f:X\to X$ and $n\ge1$, the $n$-fold product $f^{(n)}:X^{n}\to X^n$ is defined by $f^{(n)}(\zeta)=(f(x_1),f(x_2),\dots,f(x_n))$ for $\zeta=(x_1,x_2,\dots,x_n)\in X^n$. Then, we say that a compact metric space $(X,d)$ has the {\em property*} when for any $\epsilon>0$, there is $\delta>0$ such that the following condition holds: Given any integer $n\ge1$ and any pair of proper $n$-tuples $\zeta,\eta\in X^n$, if $d_n(\zeta,\eta)<\delta$, then there is $\phi\in\mathcal{H}(X)$ such that $D(\phi,id_X)<\epsilon$ and $\phi^{(n)}(\zeta)=\eta$.

A compact metric space $(X,d)$ is said to be a {\em Cantor space} if it is perfect, that is, it has no isolated point, and its topological dimension (denoted by $\dim X$) is zero, or equivalently, it is totally disconnected. Every Cantor space is homeomorphic  to the Cantor ternary set in the unit interval. We let the following lemma be a base of the study in this paper.

\begin{lem}
Any Cantor space $(X,d)$ has the property*.
\end{lem}

Under the assumption that the space is perfect and has the property*, we prove that a topologically stable homeomorphism satisfies not only the shadowing property but also the strict periodic shadowing property.

\begin{thm}
Let $(X,d)$ be a compact metric space which is perfect and has the property*. For any $f\in\mathcal{H}(X)$, if $f$ is topologically stable, then $f$ has the shadowing property and the strict periodic shadowing property, especially $f$ satisfies $CR(f)=\overline{Per(f)}$.
\end{thm}

When the property* is absent, the implication as in Theorem 1.1 does not hold in general.  For example in \cite{C}, Cook gave an example of a non-degenerate continuum $C$ with $\mathcal{H}(C)=\{id_C\}$. For such $C$, $id_C$ is trivially topologically stable, but we easily see that it has neither the shadowing property nor the strict periodic shadowing property.

It is known that there exists a circle homeomorphism which satisfies the shadowing property but is not topologically stable \cite{Y}. The {\em continuous shadowing property} introduced by Lee in \cite{L} is a stronger property than the shadowing property (its precise definition is given in Section 2). The following statement is a consequence of \cite[Theorem 2.5]{L}: If $f\in\mathcal{H}(M)$ has the continuous shadowing property, then $f$ is topologically stable. There seems to be an implicit assumption in \cite[Theorem 2.5]{L} that the space $M$ is a closed differentiable manifold. In contrast, by using Theorem 1.1, we prove the following corollary.

\begin{cor}
Let $(X,d)$ be a Cantor space. Then, there exists $f\in\mathcal{H}(X)$ which has the continuous shadowing property but is not topologically stable.
\end{cor}

Indeed, Corollary 1.1 is an immediate consequence of the fact that the odometers satisfy the continuous shadowing property, but by Theorem 1.1, they are not topologically stable. This corollary clarifies that the continuous shadowing property does not necessarily imply the topological stability unless there are proper assumptions on the space.

The next result concerns the notion of equicontinuity. For $f\in\mathcal{H}(X)$, we say that $f$ is {\em equicontinuous} if for any $\epsilon>0$, there is $\delta>0$ such that $d(x,y)\le\delta$ implies $\sup_{i\in\mathbb{Z}}d(f^i(x),f^i(y))\le\epsilon$ for all $x,y\in X$.

\begin{thm}
Let $(X,d)$ be a compact metric space and let $f\in\mathcal{H}(X)$ be an equicontinuous homeomorphism. If $f$ has the strict periodic shadowing property, then $f$ is topologically stable, and $\dim X=0$.
\end{thm}

This theorem gives a converse of Theorem 1.1 for equicontinuous homeomorphisms. As a direct consequence of Lemma 1.1, Theorem 1.1, and Theorem 1.2, we obtain the following corollary.

\begin{cor}
Let $(X,d)$ be a Cantor space and let $f\in\mathcal{H}(X)$ be an equicontinuous homeomorphism. Then, $f$ is topologically stable if and only if $f$ has the strict periodic shadowing property.
\end{cor}

As proved in Section 5 (see Lemma 5.1), for any equicontinuous homeomorphism $f\in\mathcal{H}(X)$, if $\dim X=0$ and $X=\overline{Per(f)}$, then $f$ has the periodic shadowing property. Then, it is natural to expect that the same conditions still imply the {\em strict} periodic shadowing property. However, this is not the case. In Section 5, by modifying an odometer, we give an example of a homeomorphism $f$ of a Cantor space $(X,d)$ with the following properties.
\begin{itemize}
\item[(1)] $X=\overline{Per(f)}$.
\item[(2)] $f$ is equicontinuous.
\item[(3)] $f$ has the periodic shadowing property.
\item[(4)] $f$ does not have the strict periodic shadowing property.
\item[(5)] $f^3$ has the strict periodic shadowing property.
\end{itemize}
As a consequence, this example shows that the periodic shadowing property is not equivalent to the strict periodic shadowing property in general. It also shows that even if $f^n$ has the strict periodic shadowing property for some $n>0$, $f$ does not necessarily have the same property. Moreover, by the properties (2), (4), (5), and Corollary 1.2, we obtain the following corollary. 

\begin{cor}
Let $(X,d)$ be a Cantor space. Then, there exists $f\in\mathcal{H}(X)$ such that $f^3$ is topologically stable, but $f$ is not topologically stable.
\end{cor}

It seems to be a natural attempt to seek a kind of shadowing property which is equivalent to the topological stability, and such a shadowing property may be expected to have the property that $f$ has the property iff  $f^n$ has the property for some $n>0$, as the standard shadowing property does. However, Corollary 1.3 shows that such an attempt should fail in the Cantor spaces. We remark that when $X=S^1$, it is known that $f\in\mathcal{H}(S^1)$ is topologically stable iff $f$ is topologically conjugate to a Morse-Smale diffeomorphism \cite{Y}, so it holds that $f$ is topologically stable iff $f^n$ is so for some $n>0$ iff $f^n$ is so for all $n>0$.

As a complement to Theorem 1.2, we prove the following statement.

\begin{prop}
Let $(X,d)$ be a compact metric space and let $f\in\mathcal{H}(X)$. If $\dim X=0$ and $X=Per(f)$, then $f$ is equicontinuous and satisfies the strict periodic shadowing property. 
\end{prop}

Note that in Proposition 1.1, not only dense but {\em every} point of $X$ is assumed to be a periodic point for $f$. In fact, this proposition implies that if a homeomorphism $f\in\mathcal{H}(X)$ of a Cantor space $(X,d)$ has the properties (2) and (4) above, then $X$ must contain a compact $f$-invariant subset $S\subset X$ such that  $f|_S$ is topologically conjugate to an odometer (see Lemma 2.2 in Section 2). By Proposition 1.1 and Theorem 1.2, we obtain the following corollary.

\begin{cor}
Let $(X,d)$ be a compact metric space. If $\dim X=0$, then the identity map $id_X:X\to X$ is topologically stable.
\end{cor}

This corollary shows that when the dimension of the space is zero, topologically stable homeomorphisms may exhibit a very non-hyperbolic behavior. Indeed, there is a conjecture that if a homeomorphism of a closed topological manifold is topologically stable, then its restriction to the non-wandering set is expansive (see \cite[Remark 2.4.12]{AH}). This conjecture strictly fails when the space is totally disconnected.

The subsequent theorems are general results on the topological stability of the homeomorphisms of Cantor spaces. Here, let us remark that the conjugating map $h\in\mathcal{C}(X)$ which will appear in the proof of Theorem 1.2 in Section 4 is degenerate in the sense that its image is a finite set, so possibly far from the whole space. When $X$ is a closed topological manifold, for sufficiently small $\epsilon>0$, $h\in\mathcal{C}(X)$ with $d_{C^0}(h,id_X)<\epsilon$ is surjective (see \cite[Remark 2.4.6]{AH}). Note that if we put an additional assumption that the map $h$ is surjective in the definition of the topological stability, then for example, the identity map $id_C$ of the Cantor ternary set $C$ is not topologically stable. It is worth mentioning that such a phenomenon was already observed by Walters in \cite{W}. In \cite{W}, it was proved that for any compact metric space $(X,d)$ and $f\in\mathcal{H}(X)$, if $f$ is expansive and has the shadowing property, then $f$ is topologically stable, and also if $\epsilon>0$ is sufficiently small, the conjugating map $h\in\mathcal{C}(X)$ with $d_{C^0}(h,id_X)<\epsilon$ must be unique. Hence, for example, the full shift $\sigma$ on two symbols is topologically stable, but there is a periodic homeomorphism $g$, i.e., $g^m=id$ for some $m>0$, which is arbitrary close to $\sigma$, so for such $g$, the image of the unique map $h$ near $id$ with $h\circ g=\sigma\circ h$ must be a finite set. Taking into account these observations, we make the following definition.

\begin{defi}
For $f\in\mathcal{H}(X)$, $f$ is said to be topologically stable in the strong sense (or s-topologically stable) if for any $\epsilon>0$, there is $\delta>0$ such that for every $g\in\mathcal{H}(X)$ with $D(f,g)<\delta$, there is a surjective $h\in\mathcal{C}(X)$ with $d_{C^0}(h,id_X)<\epsilon$ and $h\circ g=f\circ h$.
\end{defi}

In this definition, the full shift on two symbols and the identity map of the Cantor ternary set are not s-topologically stable. Then, it would be natural to ask even whether there exists an s-topologically stable homeomorphism of a Cantor space or not. The following theorem states that such a homeomorphism should have quite simple dynamics as the Morse-Smale diffeomorphisms.

\begin{thm}
Let $(X,d)$ be a Cantor space. If $f\in\mathcal{H}(X)$ is topologically stable in the strong sense, then $f$ has the shadowing property, and $\Omega(f)$ is a finite set.
\end{thm}

In the proof of Theorem 1.3, we use the fact that when $(X,d)$ is a Cantor space, there exists $g\in\mathcal{H}(X)$ such that its conjugacy class \[\{\phi\circ g\circ\phi^{-1}:\phi\in\mathcal{H}(X)\}\] is residual in $(\mathcal{H}(X),D)$, i.e., containing dense $G_\delta$ set (see \cite{AGW, KR}). As a simple application of Theorem 1.1, we prove here that such $g$ is not topologically stable, so the following theorem holds.

\begin{thm}
Let $(X,d)$ be a Cantor space. Then, generic $f\in\mathcal{H}(X)$ is topologically unstable.
\end{thm}  

\begin{proof}
By Lemma 1.1 and Theorem 1.1, it is sufficient to prove that $Per(g)=\emptyset$. Assume the contrary, i.e., there are $p\in X$ and $m>0$ such that $g^m(p)=p$. We take a homeomorphism $F\in\mathcal{H}(X)$ such that $Per(F)=\emptyset$ (for instance, a homeomorphism which is topologically conjugate to an odometer). Then, since $\{\phi\circ g\circ\phi^{-1}:\phi\in\mathcal{H}(X)\}$ is dense in $(\mathcal{H}(X),D)$, there are $\phi_n\in\mathcal{H}(X)$, $n\ge1$, such that \[D(\phi_n\circ g\circ\phi_n^{-1},F)\to0\] as $n\to\infty$. For each $n\ge 1$, put $g_n=\phi_n\circ g\circ\phi_n^{-1}$ and $p_n=\phi_n(p)$. Then, we have $D(g_n^m,F^m)\to0$ as $n\to\infty$, and $g_n^m(p_n)=p_n$ for all $n\ge1$. By the compactness of $X$, there is an increasing sequence of integers $1\le n_1<n_2<\cdots$ such that $\lim_{k\to\infty}d(p_{n_k},p)=0$ for some $p\in X$. From \[d(F^m(p_{n_k}),p_{n_k})\le d(F^m(p_{n_k}),g_{n_k}^m(p_{n_k}))+d(g_{n_k}^m(p_{n_k}),p_{n_k})\le D(F^m,g_{n_k}^m)\] and $D(F^m,g_{n_k}^m)\to0$ as $k\to\infty$, it follows that $d(F^m(p),p)=\lim_{k\to\infty}d(F^m(p_{n_k}),p_{n_k})=0$, but this contradicts that $Per(F)=\emptyset$.
\end{proof}

This paper consists of six sections. Some preliminaries are given in Section 2. In Section 3, we prove Lemma 1.1, Theorem 1.1, and Corollary 1.1. We prove Theorem 1.2 and Proposition 1.1 in Section 4. In Section 5, the example of a homeomorphism of a Cantor space with the five properties listed above is given. Finally, we prove Theorem 1.3 in Section 6.

\section{Preliminaries}

In this section, we give some notations, definitions, and results used in this paper. As mentioned in Section 1, we prove the following lemma for the sake of completeness.

\begin{lem}
The two metrics $d_{C^0}$ and $D$ are equivalent metrics on $\mathcal{H}(X)$.
\end {lem}

\begin{proof}
It is obvious from the definition that for any $\epsilon>0$, $D(f,g)<\epsilon$ implies $d_{C^0}(f,g)<\epsilon$ for all $f,g\in\mathcal{H}(X)$. Conversely, let us prove that for any given $f\in\mathcal{H}(X)$ and $\epsilon>0$, there is $\delta>0$ such that every $g\in\mathcal{H}(X)$ with $d_{C^0}(f,g)<\delta$ satisfies $D(f,g)<\epsilon$. Take $0<\delta<\epsilon$ so small that $d(x,y)<\delta$ implies $d(f^{-1}(x), f^{-1}(y))<\epsilon$ for all $x,y\in X$, and suppose $g\in\mathcal{H}(X)$ satisfies $d_{C^0}(f,g)<\delta$. Then, for any $x\in X$, we have \[d(f(g^{-1}(x)),x)=d(f(g^{-1}(x)), g(g^{-1}(x)))<\delta,\] so by the choice of $\delta$, \[d(f^{-1}(x), g^{-1}(x))=d(f^{-1}(x), f^{-1}(f(g^{-1}(x))))<\epsilon.\] This implies $d_{C^0}(f^{-1},g^{-1})<\epsilon$ and so $D(f,g)<\epsilon$.
\end{proof}

\subsection{{\it Recurrence}}

Let $f\in\mathcal{H}(X)$. A point  $x\in X$ is said to be {\em periodic} if $f^n(x)=x$ for some $n>0$, {\em regularly recurrent} if for any neighborhood $U$ of $x$, there is $k>0$ such that $f^{kn}(x)\in U$ for all $n\ge0$, {\em minimal} (or {\em almost periodic}) if the restriction of $f$ to the orbit closure $\overline{O_f(x)}=\overline{\{f^n(x):n\ge0\}}$ is minimal, and {\em non-wandering} if for every neighborhood $U$ of $x$, we have $f^n(U)\cap U\ne\emptyset$ for some $n>0$. We denote by $Per(f)$, $RR(f)$, $M(f)$, and $\Omega(f)$ the sets of periodic, regularly recurrent, minimal, and non-wandering points for $f$, respectively. As mentioned in Section 1, we also denote by $CR(f)$ the set of chain recurrent points for $f$. Note that $Per(f)\subset RR(f)\subset M(f)\subset\Omega(f)\subset CR(f)$.

\subsection{{\it Continuous shadowing property}}

Let $(X,d)$ be a compact metric space and let $X^\mathbb{Z}$ be the set of all bi-infinite sequences of points in $X$. We define a metric $\tilde{d}$ on  $X^\mathbb{Z}$ by
\begin{equation*}
\tilde{d}(x,y)=\sup_{i\in\mathbb{Z}}2^{-|i|}d(x_i,y_i)\quad\mbox{for all}\quad x=(x_i)_{i\in\mathbb{Z}},\:y=(y_i)_{i\in\mathbb{Z}}\in X^\mathbb{Z},
\end{equation*}
which is compatible with the product topology. For $f\in\mathcal{H}(X)$ and $\delta>0$, let $P(f,\delta)$ denote the subset of $X^\mathbb{Z}$ consisting of all $\delta$-pseudo orbits of $f$. Then, for $f\in\mathcal{H}(X)$, we say that $f$ has the {\em continuous shadowing property} if for any $\epsilon>0$, there are $\delta>0$ and a continuous map $r:P(f,\delta)\to X$ such that
\begin{equation*}
d(f^i(r(x)),x_i)\le\epsilon\quad\mbox{for all}\quad x=(x_i)_{i\in\mathbb{Z}}\in P(f,\delta)\:,\: i\in\mathbb{Z}.
\end{equation*}
It is easy to see that, given two compact metric spaces $(X_i,d_i)$, $i=1,2$, if $f_1\in\mathcal{H}(X_1)$ and $f_2\in\mathcal{H}(X_2)$ are topologically conjugate, then $f_1$ has the continuous shadowing property iff $f_2$ has the same property.

\subsection{{\it Equicontinuity}}

Given $f\in\mathcal{H}(X)$, we say that $f$ is {\em equicontinuous at $x\in X$} if for any $\epsilon>0$, there is $\delta>0$ such that $d(x,y)\le\delta$ implies $\sup_{i\in\mathbb{Z}}d(f^i(x),f^i(y))\le\epsilon$ for all $y\in X$. Then, $f$ is said to be {\em equicontinuous} if for any $\epsilon>0$, there is $\delta>0$ such that $d(x,y)\le\delta$ implies $\sup_{i\in\mathbb{Z}}d(f^i(x),f^i(y))\le\epsilon$ for all $x,y\in X$. By the compactness of $X$, we easily see that $f$ is equicontinuous iff $f$ is equicontinuous at every $x\in X$. Every equicontinuous homeomorphism $f\in\mathcal{H}(X)$ is known to satisfy $X=M(f)$ and so $X=CR(f)$. As for the relation with the shadowing property, we have the following lemma.

\begin{lem}
Let $(X,d)$ be a compact metric space and let $f\in\mathcal{H}(X)$ be an equicontinuous homeomorphism. Then, the following holds.
\begin{itemize}
\item[(1)] $f$ has the shadowing property if and only if $\dim X=0$.
\item[(2)] If $\dim X=0$, then $X=RR(f)$, and for any $x\in X$, $x\in Per(f)$, or $f|_{\overline{O_f(x)}}$ is topologically conjugate to an odometer.
\end{itemize}
\end{lem}

\begin{proof}
For (1), see \cite[Theorem 4]{M}. Then, we shall prove (2) as follows. In fact, Lemma 4.2 in Section 4 implies that if $f\in\mathcal{H}(X)$ is equicontinuous, and $\dim X=0$, then $X=RR(f)$. Note that, as $RR(f)\subset M(f)$, $f|_{\overline{O_f(x)}}$ is minimal for every $x\in X$. On the other hand, due to \cite[Corollary 2.5]{BK}, we know that for any continuous self-map  $g\in\mathcal{C}(Y)$ of a compact metric space $Y$, if $g$ is minimal, and $Y=RR(g)$, then $Y$ is a periodic orbit of $g$, or $g$ is topologically conjugate to an odometer. Thus, by taking $Y=\overline{O_f(x)}$ and $g=f|_{\overline{O_f(x)}}$, we obtain (2).
\end{proof}

\subsection{{\it Odometers}}

An {\em odometer} (also called an {\em adding machine}) is defined as follows. Given a strictly increasing sequence $m=(m_k)_{k\ge1}$ of positive integers such that $m_1\ge2$ and $m_k$ divides $m_{k+1}$ for each $k=1,2,\dots$, we define
\begin{itemize}
\item $X(k)=\{0,1,\dots,m_k-1\}$ (with the discrete topology).
\item $X_m=\{(x_k)_{k\ge1}\in\prod_{k\ge1}X(k):x_k\equiv x_{k+1}\pmod{m_k},\:\forall k\ge1\}$.
\item $g(x)_k=x_k+1\pmod{m_k}$ for all $x=(x_k)_{k\ge1}\in X_m$, $k\ge1$.
\end{itemize}
The set $X_m$ has the subspace topology induced by the product topology on $\prod_{k\ge1}X(k)$, and the resulting dynamical system $(X_m,g)$ is called an odometer with the periodic structure $m$. It is immediate from the definition that $g:X_m\to X_m$ is an equicontinuous homeomorphism, and in fact, the odometers are characterized as the minimal equicontinuous systems on Cantor spaces (see \cite{Ku}).

\section{Proof of  Lemma 1.1, Theorem 1.1, and Corollary 1.1}

In this section, we prove Lemma 1.1, Theorem 1.1, and Corollary 1.1. We first prove Lemma 1.1. Here, note that the property* is a topological property, that is, if two compact metric spaces $(X_1,d_1)$ and $(X_2,d_2)$ are homeomorphic, then $(X_1,d_1)$ has the property* iff $(X_2,d_2)$ has the property*. Therefore, it is sufficient to prove the following lemma.

\begin{lem}
Let $X=\{0,1\}^\mathbb{N}$ and let $d:X\times X\to[0,\infty)$ be the metric defined by
\begin{equation*}
d(x,y)=\sup_{i\ge1}2^{-i}|x_i-y_i|\quad\mbox{for all}\quad x=(x_i)_{i\ge1},\:y=(y_i)_{i\ge1}\in X.
\end{equation*}
Given $\delta>0$ and an integer $n>0$, suppose that two $n$-tuples $\zeta=(x_1,x_2,\dots,x_n), \eta=(y_1,y_2,\dots,y_n)\in X^n$ are both proper and $d_n(\zeta,\eta)<\delta$. Then, there exists $\phi\in\mathcal{H}(X)$ such that $D(\phi,id_X)<\delta$ and $\phi^{(n)}(\zeta)=\eta$.
\end{lem}

Some notations are needed for the proof. For $w\in\{0,1\}^K$ with $K\ge0$, we denote by $[w]$ the following subset of $\{0,1\}^\mathbb{N}$.
\begin{equation*}
[w]=\{(x_i)_{i\ge 1}\in\{0,1\}^\mathbb{N}: x_i=w_i,\: 1\le\forall i\le K\}\subset\{0,1\}^\mathbb{N}.
\end{equation*}
When $K=0$, $w$ is defined to be the empty word $\varepsilon$, and $[\varepsilon]=\{0,1\}^\mathbb{N}$.

For $x=(x_i)_{i\ge1},\:y=(y_i)_{i\ge1}\in\{0,1\}^\mathbb{N}$, we define $z=x+y\in\{0,1\}^\mathbb{N}$ by
\begin{equation*}
z_i=x_i+y_i\pmod{2}\quad \mbox{for each}\: i\ge1.
\end{equation*}

\begin{proof}[Proof of Lemma 3.1]
First, we consider the case where $x_i\ne y_j$ for all $1\le i, j\le n$. We may assume  $\delta\in(0,1]$ and take an integer $m\ge0$ such that $\delta\in(2^{-m-1},2^{-m}]$. By $d_n(\zeta,\eta)<\delta$ and the definition of the metric $d$, we have $d(x_i,y_i)\le 2^{-m-1}$ for every $1\le i\le n$.  Since $x_1,x_2,\dots,x_n,y_1,y_2,\dots,y_n$ are pairwise distinct points in $X$, for sufficiently large $K\ge1$, we can choose $2n$-distinct elements $\alpha_1,\alpha_2,\dots,\alpha_n,\beta_1,\beta_2,\dots,\beta_n\in\{0,1\}^K$ such that $x_i\in[\alpha_i]$ and $y_i\in[\beta_i]$ for every $1\le i\le n$. Then, for each $1\le i\le n$, since $d(x_i,y_i)\le 2^{-m-1}$, there are $s_i\in\{0,1\}^m$ and $t_i, u_i\in\{0,1\}^{K-m}$ for which $\alpha_i=s_it_i$ and $\beta_i=s_iu_i$. For each $1\le i\le n$, we define $w_i=(w_{i,j})_{j\ge1}\in\{0,1\}^\mathbb{N}$ by 
\begin{equation*}
w_{i,j}=
\begin{cases}
0 &\text{if $x_{i,j}=y_{i,j}$}\\
1 &\text{if $x_{i,j}\ne y_{i,j}$}
\end{cases}
\end{equation*}
for every $j\ge1$, implying that $y_i=x_i+w_i$. Then, we define $\phi:X\to X$ by 
\begin{equation*}
\phi(x)=
\begin{cases}
x+w_i &\text{if $x\in[\alpha_i]\cup[\beta_i]$ for $1\le i\le n$}\\
x &\text{if $x\in[\gamma]$ for $\gamma\in\{0,1\}^K\setminus\{\alpha_1,\alpha_2,\dots,\alpha_n,\beta_1,\beta_2,\dots,\beta_n\}$}
\end{cases}.
\end{equation*}
By the definition, it is clear that $\phi(x_i)=y_i$ for each $1\le i\le n$, i.e., $\phi^{(n)}(\zeta)=\eta$. Note that $\phi([\alpha_i])=[\beta_i]$, $\phi([\beta_i])=[\alpha_i]$ for each $1\le i\le n$. Then, it is also clear that $\phi$ is continuous and $\phi^2=id_X$, therefore $\phi\in\mathcal{H}(X)$. For all $1\le i\le n$ and $x\in[\alpha_i]$, we have 
\begin{equation*}
d(\phi(x),x)\le{\rm diam}[s_i]\le2^{-m-1}<\delta
\end{equation*}
because $\phi(x)\in[\beta_i]$ and $[\alpha_i]\cup[\beta_i]\subset[s_i]$. Similarly, we have $d(\phi(x),x)<\delta$ for all $1\le i\le n$ and $x\in[\beta_i]$. Hence, it holds that $d_{C^0}(\phi, id_X)=d_{C^0}(\phi^{-1}, id_X)<\delta$, so we have $D(\phi,id_X)<\delta$.

Now, let us give the proof for the general case without the assumption above. By $d_n(\zeta,\eta)<\delta$, we can choose $d_n(\zeta,\eta)<\rho<\delta$ and $\epsilon>0$ so that $\rho+2\epsilon<\delta$. Since $(X,d)$ is perfect, we can take an $n$-tuple $\theta=(z_1,z_2,\dots,z_n)\in X^n$ satisfying the following properties.
\begin{itemize}
\item[(1)] $x_1,x_2,\dots,x_n,z_1,z_2,\dots,z_n$ are pairwise distinct, and $y_1,y_2,\dots,y_n,z_1,z_2,\dots,z_n$ are also pairwise distinct points in $X$.
\item[(2)] $d_n(\zeta,\theta)<\epsilon$.
\end{itemize}
Note that we have $d_n(\theta,\eta)\le d_n(\zeta,\eta)+d_n(\zeta,\theta)<\rho+\epsilon$. Then, as it is already proved, there are $\phi_1,\phi_2\in\mathcal{H}(X)$ such that $\phi_1^{(n)}(\zeta)=\theta$, $\phi_2^{(n)}(\theta)=\eta$, $D(\phi_1,id_X)<\epsilon$, and $D(\phi_2,id_X)<\rho+\epsilon$. Let $\phi=\phi_2\circ \phi_1\in\mathcal{H}(X)$. Then, we have $\phi^{(n)}(\zeta)=\eta$, and $D(\phi,id_X)<\rho+2\epsilon<\delta$, which proves the lemma.
\end{proof}

Then, we give a proof of Theorem 1.1. As mentioned in Section 1, it was essentially proved by Walters in \cite{W} that, under the assumption of Theorem 1.1, topologically stable homeomorphisms satisfy the shadowing property, so we may omit its proof. In the following proof, we prove that those homeomorphisms satisfy the strict periodic shadowing property by a modification of Walters' argument.

\begin{proof}[Proof of Theorem 1.1]
For any given $\epsilon>0$, we take $\beta>0$, $\gamma>0$,  $\delta>0$, and $\rho>0$ with the following properties.
\begin{itemize}
\item[(1)] Given $g\in\mathcal{H}(X)$, if $D(f,g)<\beta$, then there is $h\in\mathcal{C}(X)$ such that $d_{C^0}(h, id_{X})<\epsilon/2$ and $h\circ g=f\circ h$.
\item[(2)] For any $g\in\mathcal{H}(X)$, $d_{C^0}(f,g)<\gamma$ implies $D(f,g)<\beta$.
\item[(3)] For $\gamma>0$, $\delta>0$ is so small as in the definition of the property*.
\item[(4)] $0<\rho<\min\{\delta/3,\epsilon/2\}$, and $d(a,b)<\rho$ implies $d(f(a),f(b))<\delta/3$ for all $a,b\in X$.
\end{itemize}

Suppose that $(x_i)_{i=0}^m$ is a $\delta/3$-cycle of $f$, and let us prove that there is $p\in X$ satisfying $f^m(p)=p$ and $d(f^i(p),x_i)\le\epsilon$ for all $0\le i\le m$. Put $x=x_0=x_m$ and $\zeta=(x_0,x_1,\dots,x_m)\in X^{m+1}$. Then, since $(X,d)$ is perfect, we can take an $(m+1)$-tuple $\eta=(z_0,z_1,\dots,z_m)\in X^{m+1}$ with the following properties.

\begin{itemize}
\item[(5)] $z_0=z_m=x$ and $d_{m+1}(\zeta,\eta)<\rho$. 
\item[(6)] $(y_1,y_2,\dots,y_m)=(f(z_0),f(z_1),\dots,f(z_{m-1}))$, $(z_1,z_2,\dots,z_m)\in X^m$ are both proper $m$-tuples.
\end{itemize}
Then, by (4) and (5), we have
\begin{equation*}
\begin{aligned}
d(z_i,y_i)&=d(z_i,f(z_{i-1}))\\
&\le d(z_i,x_i)+ d(x_i,f(x_{i-1}))+d(f(x_{i-1}),f(z_{i-1}))<\rho+\delta/3+\delta/3<\delta
\end{aligned}
\end{equation*}
for every $1\le i\le m$. By (3) and (6), there exists $\phi\in\mathcal{H}(X)$ such that $D(\phi,id_X)<\gamma$, and $\phi(y_i)=z_i$ for every $1\le i\le m$. Put $g=\phi\circ f\in\mathcal{H}(X)$ and note that we have
\begin{equation*}
g(z_i)=\phi(f(z_i))=\phi(y_{i+1})=z_{i+1}
\end{equation*}
for every $0\le i<m$. This implies $g^i(z_0)=z_i$ for each $0\le i\le m$. Since $d_{C^0}(f,g)<\gamma$, by (1) and (2), there is $h\in\mathcal{C}(X)$ such that $d_{C^0}(h, id_{X})<\epsilon/2$ and $h\circ g=f\circ h$. Put $p=h(z_0)$. Then, since $z_0=z_m$ by (5), we have
\begin{equation*}
f^m(p)=f^m(h(z_0))=h(g^m(z_0))=h(z_m)=h(z_0)=p.
\end{equation*}
Moreover, for every $0\le i\le m$, we have
\begin{equation*}
d(f^i(p),z_i)=d(f^i(h(z_0)),z_i)=d(h(g^i(z_0)),z_i)=d(h(z_i),z_i)<\epsilon/2,
\end{equation*}
and so (5) yields
\begin{equation*}
d(f^i(p),x_i)\le d(f^i(p),z_i)+d(z_i,x_i)<\epsilon/2+\rho<\epsilon,
\end{equation*}
which finishes the proof.
\end{proof}

To prove Corollary 1.1, we need the following simple lemma.

\begin{lem}
Let $(X,d)$ be a compact metric space and let $f\in\mathcal{H}(X)$. If $\dim X=0$ and $f$ is equicontinuous, then $f$ has the continuous shadowing property.
\end{lem}

\begin{proof}
It follows from Lemma 2.2 (1) that $f$ satisfies the shadowing property. Hence, for any given $\epsilon>0$, we can take $\gamma>0$ and $\delta>0$ with the following properties.
\begin{itemize}
\item[(1)] $0<\gamma\le\epsilon/2$, and $d(a,b)\le\gamma$ implies $\sup_{i\in\mathbb{Z}}d(f^i(a),f^i(b))\le\epsilon/2$ for all $a,b\in X$. 
\item[(2)] Every $\delta$-pseudo orbit $(x_i)_{i\in\mathbb{Z}}$ of $f$ is $\gamma$-shadowed by some $p\in X$.
\end{itemize}
Define $r:P(f,\delta)\to X$ by $r(x)=x_0$ for $x=(x_i)_{i\in\mathbb{Z}}\in P(f,\delta)$. Then, it is obvious that $r$ is continuous, and moreover, for each $x=(x_i)_{i\in\mathbb{Z}}\in P(f,\delta)$, by (1) and (2), we have 
\begin{equation*}
d(f^i(r(x)),x_i)=d(f^i(x_0),x_i)\le d(f^i(x_0),f^i(p))+d(f^i(p),x_i)\le\epsilon/2+\gamma\le\epsilon
\end{equation*}
for all $i\in\mathbb{Z}$. This shows that $r$ satisfies the required property.
\end{proof}

As the final proof of this section, we prove Corollary 1.1.

\begin{proof}[Proof of Corollary 1.1]
Take an odometer $g:X_m\to X_m$. Since $X_m$ is a Cantor space, so $\dim X_m=0$, and $g$ is equicontinuous, by Lemma 3.2, $g$ has the continuous shadowing property. By Lemma 1.1 and Theorem 1.1, if $g$ is topologically stable, we should have $CR(g)=\overline{Per(g)}$, but this is not the case because $Per(g)=\emptyset$. Hence, $g$ is not topologically stable. Now, since $X_m$ and $X$ are both Cantor spaces, there is a homeomorphism $\phi:X_m\to X$. Put $f=\phi\circ g\circ\phi^{-1}$. Then, $f$ has the same properties as $g$.
\end{proof}

\section{Proof of Theorem 1.2 and Proposition 1.1}

In this section, we prove Theorem 1.2 and Proposition 1.1. We begin with the following lemma which seems to be a `folklore'.

\begin{lem}
Let $(X,d)$ be a compact metric space and let $f\in\mathcal{C}(X)$. Then, $f$ has the pseudo periodic shadowing property if and only if $f|_{CR(f)}$ has the shadowing property.
\end{lem}

Here, for $f\in\mathcal{C}(X)$, we say that $f$ has the {\em pseudo periodic shadowing property} if for any $\epsilon>0$, there is $\delta>0$ such that for every $\delta$-cycle $(x_i)_{i=0}^m$ of $f$, there is $p\in X$ such that $d(x_i,f^i(p))\le\epsilon$ for all $0\le i\le m$. Note that $p$ is not required to be a periodic point for $f$.

It is obvious from the definition that if $f\in\mathcal{H}(X)$ has the strict periodic shadowing property, then it has the pseudo periodic shadowing property. For any equicontinuous $f\in\mathcal{H}(X)$, we have $X=M(f)=CR(f)$, so if it has the strict periodic shadowing property, then from Lemma 4.1 and Lemma 2.2 (1), it follows that $\dim X=0$, a part of the conclusion in Theorem 1.2. In this way, $\dim X=0$ is a relatively easy consequence of the equicontinuity and the strict periodic shadowing property of $f$. We remark here that $\dim X=0$ is a vital property in the proof of Theorem 1.2. However, for the above reason, even if we put it as an {\em assumption} of Theorem 1.2, the substance of Theorem 1.2 would not be lost. This is why we give only a simple reasoning of Lemma 4.1 below.

\begin{proof}[Proof of Lemma 4.1]
We know that if $f$ has the shadowing property, then so does $f|_{CR(f)}$. Its statement can be found in \cite[Theorem 3.4.2]{AH} or \cite[Lemma 1]{M}, and we could use the arguments given in there to deduce the same conclusion from the pseudo periodic shadowing property. For the converse, we could exploit the arguments given in the proof of \cite[Theorem 3.1.6]{AH} (which states that $CR(f|_{CR(f)})=CR(f)$) to show that $f$ has the pseudo periodic shadowing property.
\end{proof}

Our proof of Theorem 1.2 and Proposition 1.1 relies on the decompositions of $X$ into the equivalence classes with respect to the chain relations, which we briefly describe below. We remark that such arguments have been already used in several papers (see \cite{Ka} for details).

Suppose that $f\in\mathcal{C}(X)$ satisfies $X=CR(f)$. For $\delta>0$, we define a relation $\sim_\delta$ on $X$ as follows: Given $x,y\in X$, $x\sim_\delta y$ iff there are two $\delta$-chains  $(x_i)_{i=0}^k$ and $(y_i)_{i=0}^l$ of $f$ such that $x_0=y_l=x$ and $x_k=y_0=y$. It is obvious from the definition that $\sim_\delta$ is symmetric and transitive, and by $X=CR(f)$, we have $x\sim_\delta x$ for every $x\in X$. Hence, $\sim_\delta$ is an equivalence relation on $X$. Each equivalence class with respect to $\sim_\delta$ is called a {\em $\delta$-chain component}. By $X=CR(f)$, we can show that $x\sim_\delta f(x)$ for every $x\in X$, and $x\sim_\delta y$ for all $x,y\in X$ with $d(x,y)<\delta$; therefore, every $\delta$-chain component $C$ is clopen and $f$-invariant, i.e., $f(C)\subset C$. Then, $X$ is decomposed into finitely many $\delta$-chain components, and such a decomposition is called a {\em $\delta$-chain decomposition}.

Fix a $\delta$-chain component $C$. Note that for any $\delta$-cycle $c=(x_i)_{i=0}^n$ of $f$, if $x_i\in C$ for some $0\le i\le n$, then $x_i\in C$ for all $0\le i\le n$. In such a case, we write $c\subset C$. Set $l(c)=n$ for every $\delta$-cycle $c=(x_i)_{i=0}^n$ of $f$. Define
\begin{equation*}
{\mathcal N}=\{n\in{\mathbb N}: \mbox{$\exists$ $\delta$-cycle $c$ of $f$ with $c\subset C$ and $l(c)=n$}\}
\end{equation*}
and put
\begin{equation*}
m=\gcd{\mathcal N}=\max\{j\in{\mathbb N}: \mbox{$j|n$ for every $n\in{\mathcal N}$}\}.
\end{equation*}
Then, we define a relation  $\sim_{\delta,m}$ on $C$ as follows: Given $x,y\in C$, $x\sim_{\delta,m}y$ iff there is a $\delta$-chain $(x_i)_{i=0}^k$ of $f$ with $x_0=x$, $x_k=y$ and $m|k$. By the definition of $m$, we easily see that $\sim_{\delta,m}$ is an equivalence relation on $C$, and by $X=CR(f)$, we have $x\sim_{\delta,m}y$ for all $x,y\in C$ with $d(x,y)<\delta$. Hence, every equivalence class $D$ with respect to $\sim_{\delta,m}$ is clopen in $X$ as $C$ is so. Take any  $p\in C$ and consider $m$ points $p,f(p),\ldots,f^{m-1}(p)$. Then, it is easy to see that $C=\bigsqcup _{i=0}^{m-1}[f^i(p)]$ gives the partition of $C$ into the equivalence classes with respect to $\sim_{\delta,m}$, where $[f^i(p)]$ denotes the equivalence class containing  $f^i(p)$. Put $D_i=[f^i(p)]$, $0\le i\le m-1$, and $D_m=D_0$. Then, as shown in \cite{Ka}, we have the following properties.
\begin{itemize}
\item[(D1)] $C=\bigsqcup _{i=0}^{m-1}D_i$ and every $D_i$, $0\le i\le m-1$, is clopen in $X$.
\item[(D2)] $f(D_i)\subset D_{i+1}$ for every $0\le i\le m-1$.
\item[(D3)] Given $x,y\in D_i$ with $0\le i\le m-1$, there is $M>0$ such that for any integer $N\ge M$, there is a $\delta$-chain $(x_i)_{i=0}^k$ of $f$ in $C$ with $x_0=x$, $x_k=y$, and $k=mN$.
\end{itemize}
We call each $D_i$, $0\le i\le m-1$, a {\em $\delta$-cyclic component} of $C$, and $C=\bigsqcup _{i=0}^{m-1}D_i$ is called a {\em $\delta$-cyclic decomposition} of $C$.

Now, let $C_i$, $1\le i\le K$, be all the $\delta$-chain components and for each $1\le i\le K$, let $C_i=\bigsqcup _{j=0}^{m_i-1}D_{i,j}$ be the $\delta$-cyclic decomposition of $C_i$. We call \[\mathcal{D}(\delta)=\{D_{i,j}:1\le i\le K,\:0\le j\le m_i-1\}\] a $\delta$-cyclic decomposition of $X$. We define \[r(\delta)=\max\{{\rm diam}\:D_{i,j}:1\le i\le K,\:0\le j\le m_i-1\}.\]

Note that every equicontinuous $f\in\mathcal{H}(X)$ satisfies $X=M(f)=CR(f)$, so we may consider the $\delta$-cyclic decomposition of $X$ for every $\delta>0$. In the above notation, we prove the following lemma.

\begin{lem}
Let $(X,d)$ be a compact metric space and let $f\in\mathcal{H}(X)$ be an equicontinuous homeomorphism. If $\dim X=0$, then $\lim_{\delta\to0}r(\delta)=0$. 
\end{lem}

\begin{proof}
Let us assume $\limsup_{\delta\to0}r(\delta)>r>0$ and deduce a contradiction. Since $f$ is equicontinuous and $\dim X=0$, by Lemma 2.2 (1), $f$ satisfies the shadowing property. We take $0<\epsilon<r/2$ and $\gamma>0$ with the following properties.
\begin{itemize}
\item[(1)] $d(a,b)\le\epsilon$ implies $\sup_{i\in\mathbb{Z}}d(f^i(a),f^i(b))\le r/2$ for all $a,b\in X$.
\item[(2)] Every $\gamma$-pseudo orbit of $f$ is $\epsilon/2$-shadowed by some point of $X$.
\end{itemize}
By the assumption, we can choose $0<\delta<\gamma$ so that $r(\delta)>r$. Then, there is a component $D_{i,j}\in\mathcal{D}(\delta)$ such that ${\rm diam}\:D_{i,j}>r$. Take $x,y\in D_{i,j}$ with $d(x,y)>r$ and any $p\in D_{i,j}$. Then, by the property (D3) above, there is an integer $L>0$ and two $\delta$-chains $(x_i)_{i=0}^L$ and $(y_i)_{i=0}^L$ of $f$ such that $x_0=y_0=p$, $x_L=x$, and $y_L=y$. By (2), $(x_i)_{i=0}^L$ and $(y_i)_{i=0}^L$ are $\epsilon/2$-shadowed by some $z$ and $w$, respectively. Then, we have \[d(z,w)\le d(z,p)+d(w,p)=d(z,x_0)+d(w,y_0)\le\epsilon/2+\epsilon/2=\epsilon,\] but 
\begin{equation*}
\begin{aligned}
d(f^L(z),f^L(w))&\ge d(x,y)-d(x,f^L(z))-d(y,f^L(w))\\
&=d(x,y)-d(x_L,f^L(z))-d(y_L,f^L(w))\\
&>r-\epsilon/2-\epsilon/2>r/2.
\end{aligned}
\end{equation*}
This contradicts (1) and finishes the proof.
\end{proof}

Now, we prove Theorem 1.2.

\begin{proof}[Proof of Theorem 1.2]
If $f$ satisfies the strict periodic shadowing property, then it is obvious from the definition that $f$ has the pseudo periodic shadowing property. Note that we have $X=M(f)=CR(f)$ because $f$ is equicontinuous. Hence, $\dim X=0$ is implied by Lemma 4.1 and Lemma 2.2 (1). Given any $\epsilon>0$, take $0<\delta<\epsilon$ so small as in the definition of the strict periodic shadowing property. Then, by Lemma 4.2, there is $\gamma>0$ such that \[r(\gamma)=\max\{{\rm diam}\:D_{i,j}:1\le i\le K,\:0\le j\le m_i-1\}<\delta,\]
where $\mathcal{D}(\gamma)=\{D_{i,j}:1\le i\le K,\:0\le j\le m_i-1\}$ is the  $\gamma$-cyclic decomposition of $X$. Take $x_{i,j}\in D_{i,j}$ for all $1\le i\le K$ and $0\le j\le m_i-1$. Then, for each $1\le i\le K$, by the property (D2), \[\gamma_i=(x_{i,0},x_{i,1},\dots,x_{i,m_i-1},x_{i,0})\]
is a $\delta$-cycle of $f$, and so by the choice of $\delta$, $\epsilon$-shadowed by some $p_i\in X$ with $f^{m_i}(p_i)=p_i$. We define $h:X\to X$ by \[h(x)=f^j(p_i)\] for $x\in D_{i,j}$, $1\le i\le K$, and $0\le j\le m_i-1$. Since $X=\bigsqcup_{i=1}^K\bigsqcup_{j=0}^{m_i-1}D_{i,j}$ is a clopen partition, $h$ is well-defined, and $h\in\mathcal{C}(X)$. Since \[d(h(x),x)=d(f^j(p_i),x_{i,j})+d(x_{i,j},x)<\epsilon+\delta<2\epsilon\] for any $x\in D_{i,j}$, $1\le i\le K$, and $0\le j\le m_i-1$, we have $d_{C^0}(h,id_X)<2\epsilon$. 

Now, take $\beta>0$ so small that $d(D,D')>\beta$ for any distinct elements $D,D'\in\mathcal{D}(\gamma)$, and suppose that $g\in\mathcal{H}(X)$ satisfies $D(f,g)<\beta$. Then, for any $x\in D_{i,j}$ with $1\le i\le K$ and $0\le j\le m_i-1$, since $f(x)\in D_{i,j+1}$ by the property (D2) and $d(f(x),g(x))<\beta$, by the choice of $\beta$, we have $g(x)\in D_{i,j+1}$. This implies that \[g(D_{i,j})\subset D_{i,j+1}\] for any $1\le i\le K$ and $0\le j\le m_i-1$. Then, for any $x\in D_{i,j}$ with $1\le i\le K$ and $0\le j\le m_i-1$, we have \[f(h(x))=f(f^j(p_i))=f^{j+1}(p_i),\] and since $g(x)\in D_{i,j+1}$, \[h(g(x))=f^{j+1}(p_i)\] (note that this holds true for $j=m_i-1$ because of $f^{m_i}(p_i)=p_i$). Thus, $h\circ g=f\circ h$. Since $g\in\mathcal{H}(X)$ with $D(f,g)<\beta$ is arbitrary, $f$ is topologically stable.
\end{proof}

Finally, we prove Proposition 1.1. A definition is needed before the proof. For any $f\in\mathcal{H}(X)$, we say that $f$ is {\em distal} if $\inf_{i\in\mathbb{Z}} d(f^i(x),f^i(y))>0$ whenever $x,y\in X$ are distinct. If $f$ is equicontinuous, then $f$ is distal, but the converse does not hold in general. However, due to \cite[Corollary 1.9]{AuGW}, we know that the converse also holds when $\dim X=0$. 

\begin{proof}[Proof of Proposition 1.1]
The condition $X=Per(f)$ implies that $f$ is distal. Since $\dim X=0$, by \cite[Corollary 1.9]{AuGW}, $f$ is equicontinuous. Let $S$ denote the set of all periodic orbits of $f$. Then, by $X=Per(f)$, we have $X=\bigsqcup_{\gamma\in S}\gamma$, a disjoint union. For $\gamma\in S$, we denote by $p(\gamma)$ the least period of $\gamma$. Fix $\epsilon>0$. Then, given any $\gamma\in S$, we take $x\in\gamma$ and set
\begin{equation*}
\Delta(\gamma)=
\begin{cases}
\min\{d(f^i(x),f^j(x)):0\le i<j<p(\gamma)\}&\text{if $p(\gamma)>1$}\\
1&\text{if $p(\gamma)=1$}
\end{cases}.
\end{equation*}
By Lemma 4.2, we can choose $\delta(\gamma)>0$ so small that $r(\delta(\gamma))<\min\{\Delta(\gamma),\epsilon\}$. Let $C(\gamma)$ be the $\delta(\gamma)$-chain component containing $x$ and let $C(\gamma)=\bigsqcup_{i=0}^{m-1}D_i$, $m=m(\gamma)>0$, be the $\delta(\gamma)$-cyclic decomposition of $C(\gamma)$ with $x\in D_0$. Then, we have $m|p(\gamma)$. If $m<p(\gamma)$, then $x\ne f^m(x)$ and $\{x,f^m(x)\}\subset D_0$, so \[d(x,f^m(x))\le{\rm diam}\:D_0\le r(\delta(\gamma))<\Delta(\gamma),\] which contradicts the choice of $\Delta(\gamma)$. Hence, $m=p(\gamma)$. Note that $\gamma\subset C(\gamma)$ and $C(\gamma)$ is clopen in $X$. By  $X=\bigcup_{\gamma\in S}C(\gamma)$ and the compactness of $X$, there are $\gamma_1,\gamma_2,\dots,\gamma_k\in S$ such that $X=\bigcup_{i=1}^kC(\gamma_i)$.

Now, take $0<\delta<\min_{1\le i\le k}\delta(\gamma_i)$ and let $\mathcal{D}(\delta)=\{D_{i,j}:1\le i\le K,\:0\le j\le m_i-1\}$ be the  $\delta$-cyclic decomposition of $X$. Then, we take $\beta>0$ so small that $d(D,D')>\beta$ for any distinct elements $D,D'\in\mathcal{D}(\delta)$. We shall show that for every $\beta$-cycle $(x_j)_{j=0}^n$ of $f$, there is $p\in X$ such that $f^n(p)=p$ and $d(x_j,f^j(p))\le\epsilon$ for all $0\le j\le n$. Without loss of generality, we may assume that $x_0\in D_{1,0}$. Then, by the choice of $\beta$ and the property (D2), we have $m_1|n$ and
\begin{equation*}
x_j\in D_{1,j}\pmod{m_1} \tag{1}
\end{equation*}
for every $0\le j\le n$. Note that $D_{1,0}\cap C(\gamma_i)\ne\emptyset$ for some $1\le i\le k$. For such $i$, let $C(\gamma_i)=\bigsqcup_{j=0}^{m'-1}D'_j$ denote the $\delta(\gamma_i)$-cyclic decomposition of $C(\gamma_i)$ with $D_{1,0}\cap D'_0\ne\emptyset$. Since $\delta<\delta(\gamma_i)$, we have $D_{1,0}\subset D'_0$, $m'|m_1$, and
\begin{equation*}
D_{1,j}\subset D'_j\pmod{m'} \tag{2}
\end{equation*}
for every $0\le j\le m_1$. Take $p\in\gamma_i\cap D'_0$. Then, because $f^{m'}(p)=p$, $m'|m_1$, and $m_1|n$, we have $f^n(p)=p$. It also holds that \[f^j(p)\in D'_j\pmod{m'}\] for every $0\le j\le n$. By (1) and (2), we have \[x_j\in D'_j\pmod{m'}\] for every $0\le j\le n$, therefore by the choice of $\delta(\gamma_i)$, \[d(x_j,f^j(p))\le{\rm diam}\:D'_j\le r(\delta(\gamma_i))<\epsilon\] for all $0\le j\le n$, which finishes the proof.
\end{proof} 

\section{Example}

In this section, we give an example of a homeomorphism $f\in\mathcal{H}(X)$ of a Cantor space $(X,d)$ with the following properties.
\begin{itemize}
\item[(1)] $X=\overline{Per(f)}$.
\item[(2)] $f$ is equicontinuous.
\item[(3)] $f$ has the periodic shadowing property.
\item[(4)] $f$ does not have the strict periodic shadowing property.
\item[(5)] $f^3$ has the strict periodic shadowing property.
\end{itemize}

First, we prove the following lemma as mentioned in Section 1. This lemma implies that the property (3) above is a consequence of  the properties (1) and (2).

\begin{lem}
Let $(X,d)$ be a compact metric space and let $f\in\mathcal{H}(X)$ be an equicontinuous homeomorphism. If $\dim X=0$ and $X=\overline{Per(f)}$, then $f$ has the periodic shadowing property.
\end{lem}

\begin{proof}
Since $\dim X=0$, by Lemma 2.2 (1), $f$ satisfies the shadowing property. Given any $\epsilon>0$, we take $\delta>0$ and $\gamma>0$ with the following properties.
\begin{itemize}
\item[(1)] Every $\delta$-pseudo orbit of $f$ is $\epsilon/2$-shadowed by some point of $X.$ 
\item[(2)] $d(a,b)\le\gamma$ implies $\sup_{i\in\mathbb{Z}}d(f^i(a),f^i(b))\le\epsilon/2$ for all $a,b\in X$.
\end{itemize}
Let $\gamma=(x_i)_{i=0}^m$ be a $\delta$-cycle of $f$. By (1), there is $x\in X$ such that $d(x_i,f^i(x))\le\epsilon/2$ for all $0\le i\le m$. Then, because $X=\overline{Per(f)}$, there is $p\in Per(f)$ such that $d(x,p)\le\gamma$. This combined with (2) implies \[d(x_i,f^i(p))\le d(x_i,f^i(x))+d(f^i(x),f^i(p))\le\epsilon/2+\epsilon/2\le\epsilon\] for each $0\le i\le m$. Since $\gamma$ is arbitrary, $f$ satisfies the periodic shadowing property.
\end{proof}

In the following subsections, we present $X$ and $f$, and then confirm successively that $f$ satisfies the properties (1)-(5) above.

\subsection*{\it{1. Construction of $X$ and $f$}}

Define $m=(m_k)_{k\ge1}$ by $m_k=2^k$ for $k\ge1$ and consider the odometer $g:X_m\to X_m$. We copy the odometer structure from $X_m$ to the Cantor ternary set $C$ in the unit interval. To do so, for any closed interval $J=[a,b]$, we set $J_0=[a,\frac{2a+b}{3}]$ and $J_1=[\frac{a+2b}{3},b]$. Put $I=[0,1]$ and define $\{I_w:w\in\{0,1\}^K,K\ge0\}$ inductively by $I_\varepsilon=I$, where $\varepsilon$ is the empty word, and 
\begin{equation*}
\begin{cases}
I_{w0}=(I_w)_0 \\
I_{w1}=(I_w)_1
\end{cases}  
\end{equation*}
for $w\in\{0,1\}^K$, $K\ge0$. We denote by $\phi:\{0,1\}^\mathbb{N}\to C$ the homeomorphism defined as
\begin{equation*}
\phi(s)=\bigcap_{k\ge1}I_{s_1s_2\cdots s_k}
\end{equation*}
for $s=(s_k)_{k\ge1}\in\{0,1\}^\mathbb{N}$. Then, for $x=(x_k)_{k\ge1}\in X_m$, put $x_0=0$ and note that we have \[x_{k+1}=x_k+0\cdot2^k\quad\mbox{or}\quad x_{k+1}=x_k+1\cdot2^k\] for each $k\ge0$. We define $\psi(x)=(\psi(x)_k)_{k\ge1}\in\{0,1\}^\mathbb{N}$ by
\begin{equation*}
\psi(x)_k=2^{-k+1}(x_k-x_{k-1})\in\{0,1\}
\end{equation*}
for every $k\ge1$. Note that $\psi:X_m\to\{0,1\}^\mathbb{N}$ is a homeomorphism. Then, $\phi\circ\psi:X_m\to C$ is a homeomorphism, and we put $h=(\phi\circ\psi)\circ g\circ (\phi\circ\psi)^{-1}\in\mathcal{H}(C)$, a homeomorphism of $C$ which is topologically conjugate to $g$.

A notation is needed. For $k\ge1$ and $l\in\{0,1,\dots,2^k-1\}$, we write  as \[l=a_0\cdot2^0+a_1\cdot2^1+\cdots+a_{k-1}\cdot2^{k-1},\] where $a_0,a_1,\dots,a_{k-1}\in\{0,1\}$, and define \[J_{k,l}=I_{a_0a_1\cdots a_{k-1}}.\]

Let us construct $X$ and $f$ as a modification of $h:C\to C$. Without disturbing the odometer structure, we shall add periodic orbits in $I\setminus C$ to ensure that periodic points are dense. Let $D\subset[0,1]$ be a copy of the Cantor ternary set. For any $\alpha,\beta\in\mathbb{R}$ and  $S\subset\mathbb{R}$, \[\alpha+\beta S=\{\alpha+\beta r:r\in S\}\subset\mathbb{R}.\] For $J=[a,b]$, we define three points $J^{(c)}$, $c\in\{0,1,2\}$, in $J\setminus(J_0\cup J_1)$ by
\begin{equation*}
J^{(c)}=\frac{2a+b}{3}+(c+1)\cdot\frac{b-a}{12}.
\end{equation*}
Then, $J(c)\subset J\setminus(J_0\cup J_1)$, $c\in\{0,1,2\}$, are defined by
\begin{equation*}
J(c)=J^{(c)}+\frac{b-a}{24}\cdot D.
\end{equation*}
For $k\ge1$, $l\in\{0,1,\dots,2^k-1\}$, and $c\in\{0,1,2\}$, we put \[p^{k}_{l+c\cdot2^k}=J^{(c)}_{k,l}\in J_{k,l}\quad\text{and}\quad D^{k}_{l+c\cdot2^k}=J_{k,l}(c)\subset J_{k,l}.\] Here, $p^{k}_{l+c\cdot2^k}\in D^{k}_{l+c\cdot2^k}$. We also put $p^{k}_{3\cdot2^k}=p^k_0$ and $D^{k}_{3\cdot2^k}=D^k_0$. Define \[\Gamma_k=\bigcup\{D^{k}_{l+c\cdot2^k}:l\in\{0,1,\dots,2^k-1\}, c\in\{0,1,2\}\}\] and note that $|\{D^{k}_{l+c\cdot2^k}:l\in\{0,1,\dots,2^k-1\}, c\in\{0,1,2\}\}|=3\cdot2^k$. We set $X=C\cup\bigcup\{\Gamma_k:k\ge1\}$ and define $f:X\to X$ as follows: $f(x)=h(x)$ for $x\in C$, and for $x\in D^{k}_{l+c\cdot2^k}$, $k\ge1$, $l\in\{0,1,\dots,2^k-1\}$, and $c\in\{0,1,2\}$, \[f(x)=\alpha+x,\] where $\alpha$ is determined by \[\alpha+D^{k}_{l+c\cdot2^k}=D^{k}_{l+c\cdot2^k+1}.\] It is immediate from the definition that $f$ is a bijection, and for each $k\ge1$, every point of $\Gamma_k$ is a periodic point for $f$ with the least period $3\cdot2^k$.

Below is a list of properties needed hereafter.
\begin{itemize}
\item $h(J_{k,l}\cap C)=J_{k,l+1}\cap C\pmod{2^k}$ for all $k\ge1$ and $l\in\{0,1,\dots,2^k-1\}$.
\item For any $m\ge k\ge1$, $l\in\{0,1,\dots,2^k-1\}$, and $n\in\{0,1,\dots,2^m-1\}$, we have \[J_{k,l}\cap J_{m,n}\ne\emptyset\iff J_{m,n}\subset J_{k,l}\iff n\equiv l\pmod{2^k}.\]
\item Putting $J_k=\bigcup\{J_{k,l}:l\in\{0,1,\dots,2^k-1\}\}$, we have $\Gamma_k\subset J_k$ for each $k\ge1$. Note that for $k\ge1$, we have \[J_k=\bigcup\{I_{a_0a_1\cdots a_{k-1}}:a_0,a_1,\dots,a_{k-1}\in\{0,1\}\},\] therefore $J_1\supset J_2\supset\cdots$ and $\:\bigcap_{k\ge1}J_k=C$. By these properties, $X$ is a compact subset of $I$. We easily see that $X$ has no isolated point, and since it is a union of countably many  compact zero-dimensional subsets, we have $\dim\:X=0$, hence $X$ is a Cantor space. 
\item Since \[J_{k+1}=\bigcup\{(J_{k,l})_0\cup(J_{k,l})_1:l\in\{0,1,\dots,2^k-1\}\}\] and \[\bigcup\{D^{k}_{l+c\cdot2^k}=J_{k,l}(c):c\in\{0,1,2\}\}\subset J_{k,l}\setminus((J_{k,l})_0\cup(J_{k,l})_1)\] for each $l\in\{0,1,\dots,2^k-1\}$, we have $\Gamma_k\cap J_{k+1}=\emptyset$ for all $k\ge1$. Since $J_{k,l}$ contains a point of $\Gamma_k$ for any $k\ge1$ and $l\in\{0,1,\dots,2^k-1\}$, we have $X=\overline{\bigcup\{\Gamma_k:k\ge1\}}=\overline{Per(f)}$.
\end{itemize}

\subsection*{\it{2. $f$ has the property (2) }}

In order to show that $f$ is an equicontinuous homeomorphism, it suffices to see that $f$ is equicontinuous at every point of $X$. Here, given any $k\ge2$, we have the following properties.
\begin{itemize}
\item[(A)]$C\subset J_k$.
\item[(B)]$\Gamma_j\subset J_k$ for every $j\ge k$.
\item[(C)]$J_k\cap(\Gamma_1\cup\Gamma_2\cup\dots\cup\Gamma_{k-1})=\emptyset$.   
\end{itemize}
The property (C) clearly implies that $f$ is equicontinuous at every point in $\Gamma_1\cup\Gamma_2\cup\dots\cup\Gamma_{k-1}$. Let us claim and prove that \[f(J_{k,l}\cap X)=J_{k,l+1}\cap X\pmod{2^k}\] for all $l\in\{0,1,\dots,2^k-1\}$. Let $p\in J_{k,l}\cap X\:(\subset J_k\cap X)$. Then, by (C), we have $p\in C$ or $p\in\Gamma_m$ for some $m\ge k$. In the former case, as mentioned above, we have  $f(p)=h(p)\in J_{k,l+1}$. In the latter case, we have $p\in D^m_{n+c\cdot2^m}\in J_{m,n}$ for some $n\in\{0,1,\dots,2^m-1\}$ and $c\in\{0,1,2\}$. Then, since $J_{k,l}\cap J_{m,n}\ne\emptyset$, as mentioned above, we have $n\equiv l\pmod{2^k}$, which implies $n+1\equiv l+1\pmod{2^k}$ and so \[f(p)\in D^m_{n+c\cdot2^m+1}\subset J_{m,n+1}\subset J_{k,l+1}.\] Therefore, in both cases, we have $f(p)\in J_{k,l+1}$, proving the claim. 

Now, since $k\ge 2$ is arbitrary,  $f$ is equicontinuous at every point in $\bigcup\{\Gamma_k:k\ge1\}$. Moreover, since $J_{k,l}\cap X$ is clopen in $X$ for every $l\in\{0,1,\dots,2^k-1\}$, and \[\lim_{k\to\infty}\max\{{\rm diam}\:J_{k,l}:l\in\{0,1,\dots,2^k-1\}\}=0,\] $f$ is equicontinuous at every point in $C=\bigcap_{k\ge1}J_k$. Thus, $f$ is equicontinuous at every point of $X$.

\subsection*{\it{3. $f$ has the property (4)}}

Let us prove that $f$ does not satisfy the strict periodic shadowing property. We shall  assume the contrary and deduce a contradiction. Given $\epsilon>0$, take $\delta>0$ so small as in the definition of the property. Take a sufficiently large $k\ge1$ and $x_l\in J_{k,l}\cap X$ for each $l\in\{0,1,\dots,2^k-1\}$. Then, $(x_0,x_1,\dots,x_{2^k-1},x_0)$ is a $\delta$-cycle of $f$, so there should be $p\in X$ such that $f^{2^k}(p)=p$ and $d(x_l,f^l(p))\le\epsilon$ for each $l\in\{0,1,\dots,2^k-1\}$. However, such a periodic point $p$ has the least period $2^\alpha$ for some $\alpha\ge0$. This contradicts that $f$ have only the periodic points whose least periods are in the form of $3\cdot2^\beta$ for some $\beta\ge0$.

\subsection*{\it{4. $f$ has the property (5)}}

Finally, we shall prove that $F=f^3$ satisfies the strict periodic shadowing property. Given $\epsilon>0$, take $k\ge2$ so large that the diameter of any element of \[\{J_{k,l}:l\in\{0,1,\dots,2^k-1\}\}\] is $\le\epsilon$. Since $F$ is equicontinuous, and $\dim X=0$, by Lemma 2.2 (1), $F$ satisfies the shadowing property. By the equicontinuity of $F$, there is $\delta>0$ such that every $\delta$-pseudo orbit $(z_i)_{i\ge0}$ of $F$ is $\epsilon$-shadowed by $z_0$ itself. If $\delta>0$ is taken to be sufficiently small, it also holds that, putting 
\begin{equation*}
\begin{aligned}
S_k=&\:\{J_{k,l}:l\in\{0,1,\dots,2^k-1\}\} \\
&\cup\{D^j_{l+c\cdot2^j}:1\le j\le k-1, l\in\{0,1,\dots,2^j-1\}, c\in\{0,1,2\}\},
\end{aligned}
\end{equation*}
we have $d(A,B)>\delta$ for any distinct elements $A,B\in S_k$.

Suppose that $(x_i)_{i=0}^m$ is a $\delta$-cycle of $F$ and let us prove that there is $p\in X$ such that $F^m(p)=p$ and $d(x_i,F^i(p))\le\epsilon$ for all $0\le i\le m$. Note that \[\Gamma_j=\bigcup\{D^j_{l+c\cdot2^j}:l\in\{0,1,\dots,2^j-1\}, c\in\{0,1,2\}\}\] for $1\le j\le k-1$, and \[X\subset\bigcup S_k=J_k\cup(\Gamma_1\cup\Gamma_2\cup\cdots\cup\Gamma_{k-1}).\] Without loss of generality, we may assume that $x_0\in J_{k,0}$ or $x_0\in D^j_0$ for some $1\le j\le k-1$. In the former case, by the choice of $\delta$, we have \[x_i\in J_{k,3i}\pmod{2^k}\] for every $0\le i\le m$, and so $2^k|m$ because $x_m=x_0\in J_{k,0}$. Put $p=p^k_0\in D^k_0$ and note that $F^{2^k}(p)=f^{3\cdot2^k}(p)=p$, then $F^m(p)=p$. Since $D^k_0\subset J_{k,0}$, we have \[F^i(p)\in J_{k,3i}\pmod{2^k}\] for every $0\le i\le m$. Therefore, by the choice of $k$, $d(x_i,F^i(p))\le\epsilon$ for all $0\le i\le m$. In the latter case, note that \[f(D^j_l)=D^j_{l+1}\pmod{3\cdot2^j}\] for every $0\le l\le 3\cdot2^j-1$. By the choice of $\delta$, we have \[x_i\in D^j_{3i}\pmod{3\cdot2^j}\] for every $0\le i\le m$, and so $2^j|m$ because $x_m=x_0\in D^j_0$. Put $p=x_0\in D^j_0$ and note that $F^{2^j}(p)=f^{3\cdot2^j}(p)=p$, then $F^m(p)=p$. Again by the choice of $\delta$, it also holds that $d(x_i,F^i(p))=d(x_i,F^i(x_0))\le\epsilon$ for all $0\le i\le m$, and this finishes the proof.

\section{Proof of Theorem 1.3}

In this section, we prove Theorem 1.3. For the proof, we need the fact that when $(X,d)$ is a Cantor space, there exists $g\in\mathcal{H}(X)$ such that its conjugacy class \[\{\phi\circ g\circ\phi^{-1}:\phi\in\mathcal{H}(X)\}\] is residual in $(\mathcal{H}(X),D)$, i.e., containing dense $G_\delta$ set, and topological entropy of such $g$ (denoted by $h_{top}(g)$) is zero (see \cite{AGW, GW, KR}). We also use the fact that when $(X,d)$ is a Cantor space, the following set \[\mathcal{TA}(X)=\{F\in\mathcal{H}(X):\text{$F$ is expansive and has the shadowing property}\}\] is dense in $(\mathcal{H}(X),D)$ (see \cite{Ki,S}). For any $F\in\mathcal{TA}(X)$, we have $CR(F)=\Omega(F)$, and $\Omega(F)$ admits the spectral decomposition, that is, \[\Omega(F)=\bigsqcup_{i=1}^k\Omega_i\] where $\Omega_i$ is clopen in $\Omega(F)$ and $F$-invariant, and $F|_{\Omega_i}$ is transitive for each $1\le i\le k$. Then, each $\Omega_i$, $1\le i\le k$, is known to admit a decomposition \[\Omega_i=\bigsqcup_{j=0}^{m_i-1}D_{i,j}\] where $D_{i,j}$ is clopen in $\Omega(F)$, $F(D_{i.j})=D_{i,j+1}\pmod{m_i}$, and $F^{m_i}|_{D_{i,j}}$ is topologically mixing for every $0\le j\le m_i-1$ (see \cite[Theorem 3.1.11]{AH}). 

\begin{proof}[Proof of Theorem 1.3]
If $f$ is s-topologically stable, then it is obvious from the definition that $f$ is topologically stable, so by Lemma 1.1 and Theorem 1.1, $f$ has the shadowing property. Fix $\epsilon>0$ and take $\delta>0$ so small as in the definition of the s-topological stability. Then, by the above fact, there is $g\in\mathcal{H}(X)$ such that $D(f,g)<\delta$ and $h_{top}(g)=0$. Given such $g$, since there is a surjective $h\in\mathcal{C}(X)$ with $h\circ g=f\circ h$, we have $h_{top}(f)=0$. Since $f$ has the shadowing property, and $h_{top}(f)=0$, by \cite[Corollary 6]{M}, $f|_{\Omega(f)}$ is equicontinuous and $\dim\Omega(f)=0$. By Lemma 2.2 (2), for each $x\in\Omega(f)$, $x\in Per(f)$ or $f|_{\overline{O_f(x)}}$ is conjugate to an odometer. Then, by the other fact above, there is $F\in\mathcal{TA}(X)$ such that $D(f,F)<\delta$, and so $H\circ F=f\circ H$ for some surjective $H\in\mathcal{C}(X)$. Let $\Omega(F)=\bigsqcup_{i=1}^k\Omega_i$ be the spectral decomposition as above and put $B_i=H(\Omega_i)$ for each $1\le i\le k$. For any fixed $1\le i\le k$, $B_i$ is $f$-invariant, and $f|_{B_i}$ is transitive, so there is $x_i\in B_i$ such that $B_i=\omega(x_i,f)$. By this, $B_i$ is a periodic orbit of $f$, or $f|_{B_i}$ is conjugate to an odometer. Let $\Omega_i=\bigsqcup_{j=0}^{m_i-1}D_{i,j}$ be the decomposition as above and put $C_i=H(D_{i,0})$ ($\subset B_i$). Then, $C_i$ is $f^{m_i}$-invariant, and $f^{m_i}|_{C_i}$ is topologically mixing. It is clear that this is not possible when $f|_{B_i}$ is conjugate to an odometer, therefore $B_i$ is a periodic orbit of $f$. It remains to prove that $\Omega(f)=\bigcup_{i=1}^kB_i$. Assume that $x\in\Omega(f)\setminus\bigcup_{i=1}^kB_i$. Then, since $H$ is surjective, there is $y\in X$ such that $H(y)=x$. Then, there is $1\le i\le k$ such that \[\lim_{n\to\infty}d(F^n(y),\Omega_i)=0,\] but this implies \[\lim_{n\to\infty}d(f^n(x),B_i)=0,\] which cannot be true because $x\in Per(f)$ or $f|_{\overline{O_f(x)}}$ is conjugate to an odometer, so  $\omega(x,f)$ is disjoint from $B_i$. This proves $\Omega(f)=\bigcup_{i=1}^kB_i$ and finishes the proof.
\end{proof}

\end{document}